\documentclass[a4paper, 12pt]{amsart}
\pdfoutput=1
\usepackage{tikz}
\usepackage{mathtools}
\usetikzlibrary{automata}

\DeclarePairedDelimiter{\abs}{\lvert}{\rvert}
\newcommand{\C}{\mathbb{C}}

\newcommand{\calE}{\mathcal{E}}
\newcommand{\calT}{\mathcal{T}}

\newcommand{\E}{\mathbb{E}}
\renewcommand{\P}{\mathbb{P}}
\newcommand{\R}{\mathbb{R}}
\newcommand{\V}{\mathbb{V}}

\newtheorem{proposition}{Proposition}[section]
\newtheorem{lemma}[proposition]{Lemma}
\newtheorem{corollary}[proposition]{Corollary}
\newtheorem{theorem}[proposition]{Theorem}
\newtheorem{remark}{Remark}

\newcommand{\TODO}[1]%
{\par\fbox{\begin{minipage}{0.9\linewidth}\textbf{TODO:} #1\end{minipage}}\par}

\renewcommand{\MR}[1]{}

\title{The height of multiple edge plane trees}
\author[C.~Heuberger]{Clemens Heuberger}
\address[Clemens Heuberger]{Institut f\"ur Mathematik, Alpen-Adria-Uni\-ver\-si\-t\"at Klagenfurt,
  Universit\"atsstra\ss e 65--67, 9020 Klagenfurt, Austria}
\email{clemens.heuberger@aau.at}
\thanks{C.~Heuberger is supported by the Austrian Science Fund (FWF):
  P~24644-N26. This paper has been written while he was a visitor at Stellenbosch University.}

\author[H.~Prodinger]{Helmut Prodinger}
\thanks{H.~Prodinger is supported by an incentive grant of the National Research Foundation of South Africa.}

\author[S.~Wagner]{Stephan Wagner}
\address[Helmut Prodinger, Stephan Wagner]{Department of Mathematical Sciences, Stellenbosch University, 7602 Stellenbosch,
 South Africa}
\email{hproding@sun.ac.za}
\email{swagner@sun.ac.za}
\thanks{S.~Wagner is supported by the National Research Foundation of South Africa, grant number 70560.}

\begin{document}

\begin{abstract}
Multi-edge trees as introduced in a recent paper of Dziemia{\'n}czuk are plane
trees where multiple edges are allowed. We first show that $d$-ary multi-edge
trees where the out-degrees are bounded by $d$ are in bijection with classical
$d$-ary trees. This allows us to analyse parameters such as the height.

The main part of this paper is concerned with multi-edge trees counted by their
number of edges. The distribution of the number of vertices as well as the
height are analysed asymptotically.
\end{abstract}
\subjclass[2010]{05A16; 05A15, 05C05, 60C05}
\keywords{Multi-edge tree; plane tree; $d$-ary tree; height; limit distributions}
\maketitle

\section{Introduction}

Dziemia{\'n}czuk~\cite{Dziemianczuk:2014:enumer-raney} has introduced  a tree model based on plane (=planar) trees~\cite[p.~31]{Flajolet-Sedgewick:ta:analy}, 
which are enumerated by Catalan numbers. Instead of connecting two vertices by
one edge, in his multi-edge model, two vertices can be connected by several
edges. If one counts trees by vertices, one must somehow restrict the number of
edges in order to avoid an infinity of objects with the same number of
vertices. In \cite{Dziemianczuk:2014:enumer-raney}, the chosen restriction is
that each vertex has out-degree at most $d$, i.e., there are at most $d$ edges going out
from any vertex. 
However, if one counts trees with a given number of edges, the restriction
 with the parameter $d$ is no longer necessary. This is in contrast to the case of
 classical plane trees where the number of edges equals the number of vertices
 minus one.

In \cite{Dziemianczuk:2014:enumer-raney}, several parameters of multi-edge
trees were analysed, but some questions about the (average) height (i.e., the maximum distance from the root) of such multi-edge trees were left open. The present paper aims to close this gap.

In Section~\ref{section:bijection}, a bijection is constructed which links $d$-ary  multiple edge
trees with standard $d$-ary trees. Since the bijection is
height-preserving, and the height of $d$-ary trees is well understood, we can
resort to results by Flajolet and
Odlyzko~\cite{Flajolet-Odlyzko:1982} as well as by Flajolet, Gao, Odlyzko and
Richmond\cite{Flajolet-Gao-Odlyzko-Richmond:1993} and provide in this way a full analysis of
the height of $d$-ary multi-edge trees, cf. Theorem~\ref{theorem:local-limit-theorem-d-ary-trees}.

In Section~\ref{section:by-edges}, we count trees by the number of edges and drop the parameter
$d$. The analysis of the height of plane trees appears in a classic paper by de
Bruijn, Knuth and Rice~\cite{Bruijn-Knuth-Rice:1972} (see also \cite{Prodinger:1983:height-planted}), with an average height of
asymptotically $\sqrt{\pi n}$. Now, we can follow this approach to some extent, but
combine it with a technique presented in
\cite{Flajolet-Prodinger:1986:regis}. The expected height is asymptotically
equal to $\frac{2}{\sqrt{5}}\sqrt{\pi n}$, with a more precise result in
Theorem~\ref{theorem:expected-height}. The constant is smaller, which is also
intuitive, since the multiple edges contribute to the size of the objects, but
not to the height. We also give an exact counting formula in terms of
weighted trinomial coefficients (Theorem~\ref{theorem:explicit-formula-height}) and a local limit theorem (Theorem~\ref{theorem:local-limit-theorem}).

The distribution of the number of vertices in plane multi-edge trees with $n$ edges is analysed
in Theorem~\ref{theorem:number-vertices}. The number of trees
with given number of vertices and edges is given in
Theorem~\ref{theorem:fixed-vertices}.

\section{A bijection between \texorpdfstring{$d$}{d}-ary multi-edge trees and ordinary \texorpdfstring{$d$}{d}-ary trees}\label{section:bijection}

As explained in the introduction, Dziemia{\'n}czuk \cite{Dziemianczuk:2014:enumer-raney} studies $d$-ary
multi-edge trees, where a vertex can have at most $d$ edges going out from
it. We present a simple bijection to ordinary (pruned) $d$-ary trees, where
every vertex has $d$ possible positions for an edge to be attached (e.g., left,
middle, right in the case $d = 3$). See
\cite[Example~I.14]{Flajolet-Sedgewick:ta:analy} for a discussion of pruned
$d$-ary trees. This bijection preserves (amongst other parameters, such as the number of leaves) the height, allowing us to reduce the problem of enumerating $d$-ary multi-edge trees by height to the analogous question for $d$-ary trees, which has been settled in \cite{Flajolet-Gao-Odlyzko-Richmond:1993}.

\medskip

Our bijection can be described as follows: suppose that a vertex $v$ of a
$d$-ary multi-edge tree has $r$ children, which are connected to $v$ by $k_1$,
$k_2$, $\ldots$, $k_r$ edges respectively. The corresponding vertex $v'$ in the $d$-ary tree also has $r$ children (corresponding to the children of $v$ in the natural way), which are attached to $v'$ by edges in the $k_1$-th, $(k_1+k_2)$-th, $(k_1+k_2+k_3)$-th,  \ldots, $(k_1+k_2+\cdots+k_r)$-th position. Since we are assuming that $k_1+k_2+\cdots+k_r$ is always $\leq d$, this is possible, and clearly this process is bijective for each vertex, so it also describes a bijection between trees. Figures~\ref{fig:5multi} and~\ref{fig:5ary} illustrate an example in the case $d=5$.

\begin{figure}[htbp]
\begin{center}
\begin{tikzpicture}
\node[fill=black,circle,inner sep=1.5pt] (v1) at (0,0) {};
\node[fill=black,circle,inner sep=1.5pt] (v2) at (-2,-1) {};
\node[fill=black,circle,inner sep=1.5pt] (v3) at (2,-1) {};
\node[fill=black,circle,inner sep=1.5pt] (v4) at (-3,-2) {};
\node[fill=black,circle,inner sep=1.5pt] (v5) at (-1,-2) {};
\node[fill=black,circle,inner sep=1.5pt] (v6) at (1,-2) {};
\node[fill=black,circle,inner sep=1.5pt] (v7) at (2,-2) {};
\node[fill=black,circle,inner sep=1.5pt] (v8) at (3,-2) {};
\node[fill=black,circle,inner sep=1.5pt] (v9) at (-3.5,-3) {};
\node[fill=black,circle,inner sep=1.5pt] (v10) at (-2.5,-3) {};
\node[fill=black,circle,inner sep=1.5pt] (v11) at (-1,-3) {};
\node[fill=black,circle,inner sep=1.5pt] (v12) at (2,-3) {};

\draw (v1)--(v2);
\draw (v1)--(v3);
\draw (v1) edge [bend left] (v3);
\draw (v1) edge [bend right] (v3);
\draw (v2)--(v4);
\draw (v2) edge [bend left] (v5);
\draw (v2) edge [bend right] (v5);
\draw (v3)--(v6);
\draw (v3)--(v7);
\draw (v3)--(v8);
\draw (v3) edge [bend left] (v8);
\draw (v3) edge [bend right] (v8);
\draw (v4) edge [bend left] (v9);
\draw (v4) edge [bend right] (v9);
\draw (v4)--(v10);
\draw (v5)--(v11);
\draw (v7) edge [bend left] (v12);
\draw (v7) edge [bend right] (v12);
\end{tikzpicture}
\end{center}
\caption{A $5$-ary multi-edge tree.}\label{fig:5multi}
\end{figure}
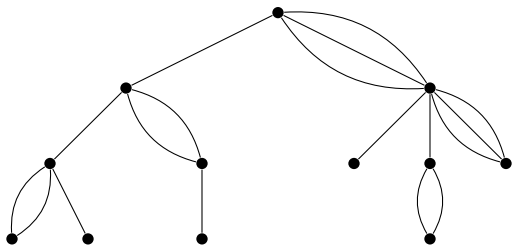

\begin{figure}[htbp]
\begin{center}
\begin{tikzpicture}
\node[fill=black,circle,inner sep=1.5pt] (v1) at (0,0) {};
\node[fill=black,circle,inner sep=1.5pt] (v2) at (-2,-1) {};
\node[fill=black,circle,inner sep=1.5pt] (v3) at (1,-1) {};
\node[fill=black,circle,inner sep=1.5pt] (v4) at (-3.5,-2) {};
\node[fill=black,circle,inner sep=1.5pt] (v5) at (-2,-2) {};
\node[fill=black,circle,inner sep=1.5pt] (v6) at (-0.5,-2) {};
\node[fill=black,circle,inner sep=1.5pt] (v7) at (0.25,-2) {};
\node[fill=black,circle,inner sep=1.5pt] (v8) at (2.5,-2) {};
\node[fill=black,circle,inner sep=1.5pt] (v9) at (-4,-3) {};
\node[fill=black,circle,inner sep=1.5pt] (v10) at (-3.5,-3) {};
\node[fill=black,circle,inner sep=1.5pt] (v11) at (-3,-3) {};
\node[fill=black,circle,inner sep=1.5pt] (v12) at (-0.25,-3) {};

\draw (v1)--(v2);
\draw[gray] (v1)--+(-0.2, -0.2);
\draw[gray] (v1)--+(0, -0.2);
\draw (v1)--(v3);
\draw[gray] (v1)--+(0.4, -0.2);
\draw (v2)--(v4);
\draw[gray] (v2)--+(-0.2, -0.2);
\draw[gray] (v2)--+(0.2, -0.2);
\draw[gray] (v2)--+(0.4, -0.2);
\draw (v2)--(v5);
\draw (v3)--(v6);
\draw (v3)--(v7);
\draw (v3)--(v8);
\draw[gray] (v3)--+(0.0, -0.2);
\draw[gray] (v3)--+(0.2, -0.2);
\draw (v4)--(v9);
\draw (v4)--(v10);
\draw[gray] (v4)--+(-0.4, -0.2);
\draw[gray] (v4)--+(0.2, -0.2);
\draw[gray] (v4)--+(0.4, -0.2);
\draw (v5)--(v11);
\draw[gray] (v5)--+(-0.1, -0.2);
\draw[gray] (v5)--+(0, -0.2);
\draw[gray] (v5)--+(0.2, -0.2);
\draw[gray] (v5)--+(0.4, -0.2);
\draw (v7)--(v12);
\draw[gray] (v7)--+(-0.4, -0.2);
\draw[gray] (v7)--+(0, -0.2);
\draw[gray] (v7)--+(0.2, -0.2);
\draw[gray] (v7)--+(0.4, -0.2);
\end{tikzpicture}
\end{center}
\caption{The associated $5$-ary tree, where each vertex can have children in five different positions (far left, left, middle, right, far right).}\label{fig:5ary}
\end{figure}
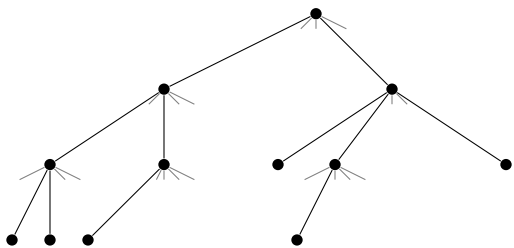

From this bijection, we immediately obtain the following corollaries:

\begin{corollary}
The number of $d$-ary multi-edge trees with $n$ vertices equals the number of $d$-ary trees with $n$ vertices, which is the Fuss-Catalan number $\frac{1}{n} \binom{nd}{n-1}$.
\end{corollary}
This is for instance shown in \cite[Example~I.14]{Flajolet-Sedgewick:ta:analy}.

\begin{corollary}
The number of $d$-ary multi-edge trees of height $h$ with $n$ vertices equals the number of $d$-ary trees of height $h$ with $n$ vertices.
\end{corollary}

It is well known that $d$-ary trees belong to the general class of \emph{simply
  generated families of trees}, and the height of such families was studied in
great detail in a paper by Flajolet, Gao, Odlyzko and Richmond
\cite{Flajolet-Gao-Odlyzko-Richmond:1993}. They obtain the following local
limit theorem (only stated for $d$-ary trees here, i.e. setting
$\phi(y)=(1+y)^d$ and $\tau=1/(d-1)$ in the formul\ae{} given there), which refines earlier results of Flajolet and Odlyzko \cite{Flajolet-Odlyzko:1982} on the average height:
\begin{theorem}[{\cite[Theorem 1.2]{Flajolet-Gao-Odlyzko-Richmond:1993}}]\label{theorem:local-limit-theorem-d-ary-trees}
Let $N_h^{(d)}(n)$ be the number of $d$-ary trees ($d$-ary multi-edge trees) with $n$ vertices whose height is $h$ and $N^{(d)}(n)$ the total number of $d$-ary trees ($d$-ary multi-edge trees) with $n$ vertices. For any $\delta > 0$, we have the asymptotic formula
\begin{align*}
\frac{N_h^{(d)}(n)}{N^{(d)}(n)} &\sim 2c\beta^4\sqrt{\pi/n} \sum_{m \geq 1} (m\pi)^2 \left(2(\pi m\beta)^2-3 \right) e^{-(\pi m\beta)^2} \\
&= 2c\beta^{-1} n^{-1/2} \sum_{m \geq 1} m^2(2(m/\beta)^2 - 3) e^{-(m/\beta)^2},
\end{align*}
where $c = \sqrt{2(d-1)/d}$, uniformly for \[\delta^{-1} (\log n)^{-1/2} \leq \beta = 2\sqrt{n}/(ch) \leq \delta (\log n)^{1/2}.\]
\end{theorem}
\begin{corollary}[{\cite[Theorem~S]{Flajolet-Odlyzko:1982}, \cite[Corollary~1.2]{Flajolet-Gao-Odlyzko-Richmond:1993}}]
The average height of $d$-ary trees with $n$ vertices (and thus also the average height of $d$-ary multi-edge trees with $n$ vertices) is asymptotically equal to \[\sqrt{2\pi d n/(d-1)}.\]
\end{corollary}

Similar results for the average height were obtained by Kemp
\cite{kemp:1980:average-height, kemp:1983:average-height} (see also \cite{Prodinger:1982:kemp-r-tuply}) for slightly different models of random plane trees, namely for trees with given root degree or number of leaves.

As it was mentioned earlier, other statistical results carry over from $d$-ary trees to $d$-ary multi-edge trees as well:

\begin{corollary}
The number of $d$-ary multi-edge trees with $n$ vertices and $k$ leaves equals the number of $d$-ary trees with $n$ vertices and $k$ leaves.
\end{corollary}

More generally, the following holds:

\begin{corollary}
For every $r \in \{0,1,\ldots,d\}$, the number of $d$-ary multi-edge trees with $n$ vertices, $k$ of which have exactly $r$ children, equals the number of $d$-ary trees with $n$ vertices, $k$ of which have exactly $r$ children. Thus the average number of vertices with exactly $r$ children is the same for $d$-ary multi-edge trees and $d$-ary trees with $n$ vertices.
\end{corollary}

It is not difficult to show that the average proportion of vertices with
exactly $r$ children is asymptotically equal to $\binom{d}{r} (d-1)^{d-r}
d^{-d}$ as $n \to \infty$ (cf.\ the paragraph following
\cite[Theorem~3.13]{Drmota:2009:random} with $k=r$, $\phi_k=\binom{d}{k}$,
$\Phi(y)=(1+y)^d$, $\tau=1/(d-1)$), which tends to $1/(r!e)$ as $d \to \infty$. This generalises the observation made in~\cite{Dziemianczuk:2014:enumer-raney} in the case $r=0$ that the asymptotic average proportion of leaves tends to $1/e$ as $d \to \infty$.

\section{Trees with Given Number of Edges}\label{section:by-edges}
In this section, we consider plane rooted multi-edge trees with a given number
$n$ of edges (which we call the \emph{size} of a tree). The resulting counting
sequence $A_n$ is sequence
\href{http://oeis.org/A002212}{A002212} in \cite{OEIS:2015}, see also
\cite{Priez:2013:lattic-combin-hopf-algeb}. It starts with 1, 1, 3, 10, 36,
137, 543, 2219, 9285, 39587.

Asymptotically, the number $A_n$ of plane rooted
multi-edge trees with $n$ edges is
\begin{equation}\label{eq:number-trees}
  A_n=\frac{5^{n+1/2}}{2\sqrt{\pi n^3}}\biggl(1+O\Bigl(\frac1n\Bigr)\biggr).
\end{equation}
This will follow without further effort at the end of the proof of Theorem~\ref{theorem:expected-height}. We now analyse the height of multi-edge trees.

\subsection{Generating Functions}
In the following lemma, we introduce the fundamental transformation which will
be used throughout this section. The principal branch of the square root function
is chosen as usual, i.e., as a holomorphic function on $\C\setminus \R_{\le 0}$
such that $\sqrt{1}=1$.
\begin{lemma}
  Let $Z=\C\setminus[1/5, 1]$ and $U=\{u\in\C \mid \abs{u}<1; u\neq
  (-3+\sqrt{5})/2\}$. Let
  \begin{align*}
    \upsilon(z)&=\frac{1-3z-\sqrt{1-5z}\sqrt{1-z}}{2z}&
    \text{for }z&\in \C,\\
    \zeta(u)&=\frac{u}{u^2+3u+1}&
    \text{for }u&\in \C\setminus\Bigl\{ \frac{-3\pm \sqrt 5}{2}\Bigr\}.
  \end{align*}

  Then $\upsilon\colon Z\to U$ and $\zeta\colon U\to Z$ are bijective
  holomorphic functions which are inverses of each other.
\end{lemma}
\begin{proof}
  We first note that $\zeta$ is well-defined and holomorphic on $U$ with
  $\zeta'(u)\neq 0$ for all $u\in U$. If $\abs{u}=1$, then
  \begin{equation*}
    \zeta(u)=\frac{1}{u+\frac1u+3}=\frac{1}{3+2\Re u}.
  \end{equation*}
  Thus the image of the unit circle under $\zeta$ is the interval $[1/5, 1]$.

  For every $z\in \C\setminus\{0\}$, $z=\zeta(u)$ is equivalent to
  \begin{equation}\label{eq:quadratic-equation-u}
    u^2+u\Bigl(3-\frac{1}{z}\Bigr)+1=0
  \end{equation}
  which has two not necessarily distinct solutions $u_1$, $u_2\in\C$ with $u_1u_2=1$. W.l.o.g.,
  $\abs{u_1}\le \abs{u_2}$. Thus either $u_1\in U$ and $\abs{u_2}>1$ or
  $\abs{u_1}=\abs{u_2}=1$. In the latter case, we have $z\in[1/5, 1]$. For
  $z=0$, $z=\zeta(u)$ is equivalent to $u=0$.  This implies that $\zeta\colon
  U\to Z$ is bijective. Furthermore, $\zeta\colon U\to Z$ has a holomorphic
  inverse $\zeta^{-1}$ defined on the simply connected region $\C\setminus[1/5,
  \infty)$.

  Solving \eqref{eq:quadratic-equation-u} explicitly yields
  \begin{equation*}
    u = \frac{1-3z\pm\sqrt{1-6z+9z^2-4z^2}}{2z}
    =\frac{1-3z\pm\sqrt{1-5z}\sqrt{1-z}}{2z}.
  \end{equation*}
  In a neighbourhood of zero, we must have $\zeta^{-1}(z)=\upsilon(z)$, because
  \begin{equation*}
    \frac{1-3z+\sqrt{1-5z}\sqrt{1-z}}{2z}
  \end{equation*}
  has a pole at $z=0$.

  It is easily seen that $\sqrt{1-5z}\sqrt{1-z}$ is a holomorphic function on
  $Z$. By the identity theorem, $\zeta^{-1}=\nu$ holds in $\C\setminus[1/5, \infty)$. By continuity of $\upsilon$ in $Z$, $\upsilon$ is also the
  inverse of $\zeta$ in $(1, \infty)$.
\end{proof}

For $h\ge 0$, consider the class $\calT_h$ of plane rooted multi-edge trees
of height at most $h$.
Denote the ordinary generating function associated to $\calT_h$ by $T_h(z)$.

\begin{lemma}The generating function $T_h(z)$ is given by
  \begin{equation}
    T_h(z) = (1-z)\frac{\alpha^{h+1}-\beta^{h+1}}{\alpha^{h+2}-\beta^{h+2}}
    =(u+1)\frac{1-u^{h+1}}{1-u^{h+2}}\label{eq:T_h-explicit}
  \end{equation}
  where
  \begin{equation}\label{eq:definition-alpha-beta}
    \begin{aligned}
      \alpha &= \frac{1-z+\sqrt{1-5z}\sqrt{1-z}}{2} = \frac{u+1}{u^2+3u+1},\\
      \beta &= \frac{1-z-\sqrt{1-5z}\sqrt{1-z}}{2} = \frac{u(u+1)}{u^2+3u+1}
    \end{aligned}\end{equation}
  for $z=\zeta(u)\in Z$.
\end{lemma}
\begin{proof}
  The class $\calT_0$ consists of an isolated vertex. For $h>0$, $\calT_h$
  consists of a root and a sequence of branches of height at most $h-1$ such that each branch is attached
  by a positive number of edges to the root. If $\calE=\{e\}$ is the class of one
  edge, we can write $\calT_h$ symbolically as
  \begin{equation}\label{eq:symbolic-equation}
    \calT_h = \circ \times (\calE^+\calT_{h-1})^*.
  \end{equation}

  The symbolic equation \eqref{eq:symbolic-equation} translates to
  \begin{equation*}
    T_h(z)=\frac1{1-\frac{z}{1-z}T_{h-1}(z)}=\frac{1-z}{1-z-zT_{h-1}(z)}.
  \end{equation*}

  This may be seen as a continued fraction. To obtain an explicit expression for
  $T_h(z)$, we use the ansatz $T_h(z)=p_h(z)/q_h(z)$ with $p_0(z)=q_0(z)=1$ and
  \begin{align*}
    p_h(z)&=(1-z)q_{h-1}(z),\\
    q_h(z)&=(1-z)q_{h-1}(z)-zp_{h-1}(z).
  \end{align*}

  Eliminating $p_h(z)$ yields the second order recurrence
  \begin{equation*}
    q_h(z)=(1-z)q_{h-1}(z)-z(1-z)q_{h-2}(z).
  \end{equation*}
  The characteristic equation is
  \begin{equation*}
    Q^2-(1-z)Q+z(1-z)=0.
  \end{equation*}
  This quadratic equation has the roots $\alpha$ and $\beta$ defined in
  \eqref{eq:definition-alpha-beta}.
  This yields the explicit expressions
  \begin{equation*}
    q_h(z) = \frac{\alpha^{h+2}-\beta^{h+2}}{(1-z)(\alpha-\beta)},\qquad
    p_h(z) = \frac{\alpha^{h+1}-\beta^{h+1}}{\alpha-\beta},
  \end{equation*}
  which result in
  \begin{equation}\label{eq:T_h-z-explicit}
    T_h(z) = (1-z)\frac{\alpha^{h+1}-\beta^{h+1}}{\alpha^{h+2}-\beta^{h+2}}.
  \end{equation}

  Under the substitution $z=\zeta(u)$, we have
  \begin{equation*}
    1-z = \frac{(u+1)^2}{u^2+3u+1},\qquad
    \beta = \frac{u(u+1)}{u^2+3u+1},\qquad
    \alpha = \frac{u+1}{u^2+3u+1}.
  \end{equation*}
  Inserting this in \eqref{eq:T_h-z-explicit} yields~\eqref{eq:T_h-explicit}.
\end{proof}

Let $T$ be the generating function of all plane, rooted multi-edge trees.

\begin{lemma}
  For $z=\zeta(u)\in Z$,
  \begin{equation*}
    T(z)=\frac{\beta}{z}=u+1
  \end{equation*}
  and
  \begin{equation}\label{eq:T-T_h-substituted-u}
    (T-T_h)(z)=\frac{1-u^2}{u}\frac{u^{h+2}}{1-u^{h+2}}.
  \end{equation}
\end{lemma}
\begin{proof}
  It is clear that $T$ is the limit of $T_h$ for $h\to\infty$. As $\abs{u}<1$,
  we have $T(z)=u+1$. The expression for $T-T_h$ follows.
\end{proof}

Note that $T$ could also have been determined by removing the restriction on
$h$ in the symbolic equation and solving the resulting quadratic equation for
$T$.

\begin{lemma}
  The functions $T(z)$, $T_h(z)$ and $\sum_{h\ge 0}(T-T_h)(z)$ are analytic for
  $z\in Z$.
\end{lemma}
\begin{proof}
  By the explicit formula for $\beta$, it is clear that $T(z)$ is an analytic
  function on $Z$.

  For $u=\upsilon(z)$ and $z\in Z$, the function
  \begin{equation*}
    T_h(z)=(u+1)\frac{1-u^{h+1}}{1-u^{h+2}}
  \end{equation*}
  is clearly analytic.

  The sum $\sum_{h\ge 0}(T-T_h)(z)$ can be written as
  \begin{equation*}
    \sum_{h\ge 0}(T-T_h)(z)=\frac{1-u^2}{u}\sum_{h\ge
      0}\frac{u^{h+2}}{1-u^{h+2}}.
  \end{equation*}
  We can bound the sum by
  \begin{equation*}
    \abs[\Big]{\sum_{h\ge
      0}\frac{u^{h+2}}{1-u^{h+2}}}\le
  \frac1{1-\abs{u}^2}\sum_{h\ge 0}\abs{u}^{h+2}=\frac{\abs{u}^2}{(1-\abs{u}^2)(1-\abs u)}.
  \end{equation*}
  By the Weierstrass $M$-test,
  \begin{equation*}
    \sum_{h\ge 0}(T-T_h)(z)
  \end{equation*}
  converges uniformly on compact subsets of $U$ and is therefore analytic in
  $U$.

  The results for $z\in Z$ follow by the fact that $\upsilon(z)$ is analytic.
\end{proof}

\subsection{Explicit Formula for the Number of Trees of Given Height}

At this stage, we can compute the number of rooted plane multi-edge trees of
size $n$ and height $>h$ explicitly. Taking the difference for $h$ and $h-1$
results in a formula for the number of trees of height $h$.

\begin{theorem}\label{theorem:explicit-formula-height}
  Let $h\ge 0$. The number of rooted plane multi-edge trees of size $n$ and
  height $>h$ is
  \begin{multline}
    \label{eq:explicit-trees-of-large-height}
    \sum_{k\ge 0}\Biggl(\binom{n-1; 1, 3, 1}{n-(h+1)-(h+2)k}-2\binom{n-1; 1, 3,
    1}{n-(h+1)-(h+2)k-2}\\
    +\binom{n-1; 1, 3, 1}{n-(h+1)-(h+2)k-4}\Biggr)
  \end{multline}
  where
  \begin{equation*}
    \binom{n; 1,3,1}{k}=[v^k](1+3v+v^2)^n
  \end{equation*}
  denotes a weighted trinomial coefficient.
\end{theorem}
\begin{proof}
  By the definition of the generating functions, we have to compute
  $[z^n](T-T_h)(z)$. By Cauchy's formula, we have
  \begin{equation}\label{eq:cauchy-1}
    [z^n](T-T_h)(z)
    =\frac{1}{2\pi i}\oint_{\abs{z} \text{ small}}\frac{(T-T_h)(z)}{z^{n+1}}\,dz.
  \end{equation}
  For sufficiently small $|u|$, the index of $0$ with respect to $\zeta(u)$
  is $1$. Therefore, using the substitution $z=\zeta(u)$ and
  using Cauchy's formula again, we can rewrite
  \eqref{eq:cauchy-1} as
  \begin{multline*}
    [z^n](T-T_h)(z)\\
    \begin{aligned}
      &=\frac{1}{2\pi i}\oint_{\abs{u} \text{
          small}}\frac{(T-T_h)(\zeta(u))}{u^{n+1}}(u^2+3u+1)^{n-1}(1-u^2)\,du\\
      &= [u^n](T-T_h)(\zeta(u))(u^2+3u+1)^{n-1}(1-u^2)\\
      &=[u^n]\frac{(1-u^2)u^{h+1}}{1-u^{h+2}}(u^2+3u+1)^{n-1}(1-u^2).
    \end{aligned}
  \end{multline*}
  Expanding the denominator into a geometric series yields
  \begin{align*}
    [z^n](T-T_h)(z)&=[u^n]\sum_{k\ge 0}(1-u^2)^2u^{h+1+(h+2)k}(u^2+3u+1)^{n-1}\\
    &=\sum_{k\ge 0}[u^{n-(h+1+(h+2)k)}] (1-2u^2+u^4)(u^2+3u+1)^{n-1}\\
    &=\sum_{k\ge 0}\bigl([u^{n-(h+1+(h+2)k)}](u^2+3u+1)^{n-1} \\&\qquad\qquad- 2
    [u^{n-(h+1+(h+2)k)-2}](u^2+3u+1)^{n-1} \\&\qquad\qquad+ [u^{n-(h+1+(h+2)k)-4}](u^2+3u+1)^{n-1}\bigr).
  \end{align*}
  By the definition of $\binom{n;1,3,1}{k}$, this is exactly \eqref{eq:explicit-trees-of-large-height}.
\end{proof}

\begin{remark}
It would be possible to determine the asymptotic behaviour of the trinomial coefficients by means of the saddle point method (cf. \cite[Section 4.3.3]{Greene-Knuth:1990:mathem}) and to obtain asymptotics for the average height (Theorem~\ref{theorem:expected-height}) and the local limit theorem (Theorem~\ref{theorem:local-limit-theorem}) from that, but the calculations would be somewhat more involved.
\end{remark}

\subsection{Expected Height}
We now compute the expected height of a random rooted plane multi-edge tree of
size $n$.

\begin{theorem}\label{theorem:expected-height}
  Let $H_n$ be the height of a random rooted plane multi-edge tree of size
  $n$. Then
  \begin{equation}\label{eq:expected-height}
    \E(H_n)=\frac{2}{\sqrt{5}}\sqrt{\pi n}- \frac{3}{2}
    + O\Bigl(\frac{1}{\sqrt{n}}\Bigr).
  \end{equation}
\end{theorem}

Before proving Theorem~\ref{theorem:expected-height}, we prove a lemma
on the harmonic sum occurring in its proof.

\begin{lemma}\label{lemma:harmonic-sum-asymptotics}We have
  \begin{multline}\label{eq:harmonic-sum-asymptotics}
    \sum_{h\ge 1}\frac{u^h}{1-u^h}= -\frac{\log(1-u)}{1-u} +
    \frac{\gamma}{1-u}\\+\frac{\log(1-u)}{2}-\frac14 - \frac{\gamma}2 + O((1-u)\log(1-u))
  \end{multline}
  as $u\to 1$ with $\abs{\arg(1-u)}<\pi/3$,
  where $\gamma$ is the Euler-Mascheroni constant.
\end{lemma}
\begin{proof}
  Using the substitution $u=e^{-t}$ yields
  \begin{equation*}
    \sum_{h\ge 1}\frac{u^h}{1-u^h}=
    \sum_{h\ge 1}\frac{e^{-ht}}{1-e^{-ht}}=
    \sum_{h\ge 1}\sum_{k\ge 1}e^{-kht}=
    \sum_{m\ge 1}d(m)e^{-mt},
  \end{equation*}
  where $d(m)$ is the number of positive divisors of $m$.

  By \cite[Example~11]{Flajolet-Gourdon-Dumas:1995:mellin}, we have
  \begin{equation*}
    \sum_{m\ge 1} d(m)e^{-mt} = \frac{1}{t}(-\log t+\gamma) + \frac{1}{4} + O(t)
  \end{equation*}
  for real $t\to 0^+$. However, the same argument can also be used for
  $\abs{\arg t}<\pi/4$ because the inverse Mellin transform
  \begin{equation*}
    e^{-t}=\frac1{2\pi i}\int_{c-i\infty}^{c+i\infty} t^{-s}\Gamma(s)\, ds
  \end{equation*}
  remains valid for complex $t$ with $\abs{\arg t}<2\pi/5$ by the identity
  theorem for analytic functions; cf. \cite{Flajolet-Prodinger:1986:regis}.

  As
  \begin{equation*}
    t=-\log u=-\log(1-(1-u))=(1-u)+\frac{(1-u)^2}{2}+O((1-u)^3),
  \end{equation*}
  substituting back yields \eqref{eq:harmonic-sum-asymptotics}.
\end{proof}

\begin{proof}[Proof of Theorem~\ref{theorem:expected-height}]
  We use the well-known identity
  \begin{align*}
    \E(H_n) &= \sum_{k=0}^\infty k\P(H_n=k)=\sum_{k>h\ge 0} \P(H_n=k)=\sum_{h\ge
      0}\P(H_n >h )\\
    &=\sum_{h\ge 0}(1-\P(H_n\le h))
    = \frac{[z^n]\sum_{h\ge 0}(T-T_h)(z)}{[z^n]T(z)}.
  \end{align*}

  We intend to compute $[z^n]\sum_{h\ge 0}(T-T_h)(z)$ via singularity
  analysis. The dominant singularity is at $z=1/5$.
  To perform singularity analysis, we need the expansion of $T-T_h$ around
  $z=1/5$, corresponding to $u=1$ under the substitution $z=\zeta(u)$.

  By \eqref{eq:T-T_h-substituted-u}, we have
  \begin{align*}
    \sum_{h\ge 0}(T-T_h)(\zeta(u)) &=
    \frac{1-u^2}{u}\sum_{h\ge 2}\frac{u^h}{1-u^h}\\
    &=-(1+u) + \frac{1-u^2}{u}\sum_{h\ge 1}\frac{u^h}{1-u^h}.
  \end{align*}
  By Lemma~\eqref{eq:harmonic-sum-asymptotics}, this is
  \begin{equation}\label{eq:expectation-expansion-u}
    \begin{aligned}
      \sum_{h\ge 0}(T-T_h)(\zeta(u))&=-2\log(1-u)-(2-2\gamma)\\
      &\qquad+\frac12(1-u)+O((1-u)^2\log(1-u)).
    \end{aligned}
  \end{equation}
  We have
  \begin{align*}
    1-u &= \sqrt{5}\sqrt{1-5z}-\frac52(1-5z)+O((1-5z)^{3/2}),\\
    \log(1-u)&= \frac12\log(1-5z)+\frac12\log 5 -
    \frac{\sqrt{5}}{2}\sqrt{1-5z}+ O((1-5z)).
  \end{align*}
  Inserting this in~\eqref{eq:expectation-expansion-u} yields
  \begin{multline*}
    \sum_{h\ge 0}(T-T_h)(z)=-\log(1-5z)
    - (2-2\gamma+\log 5)\\+
    \frac{3}2\sqrt{5}\sqrt{1-5z}
    +O((1-5z)\log(1-5z))
  \end{multline*}
  for $z\to \frac15$ and $\abs{\arg(\frac15-z)}< 3\pi/ 5$, i.e. $\abs{\arg(z-\frac15)}>2\pi/5$.
  Note that the exact bounds for the arguments are somewhat arbitrary: the
  essential property of $2\pi/5$ here is that it is less than $\pi/2$. Using
  the expansions of $1-u$ and $t$ in terms of $\sqrt{1-5z}$ and of $1-u$,
  respectively, the angles are transformed accordingly, but we have to allow
  for a small error.
  By singularity analysis \cite{Flajolet-Odlyzko:1990:singul}, this yields
  \begin{equation}\label{eq:expectation-numerator}
    \begin{aligned}
      [z^n]\sum_{h\ge 0}(T-T_h)(z) &= \frac{5^n}{n}
      +\frac{3\sqrt{5}}{2}\frac{5^n}{\Gamma(-1/2)n^{3/2}}+O\biggl(5^n\frac{\log
        n}{n^2}\biggr)\\
      &=
      \frac{5^n}{n}-\frac{3\cdot 5^{n+1/2}}{4\sqrt{\pi n^3}}+O\biggl(5^n\frac{\log n}{n^2}\biggr).
    \end{aligned}
  \end{equation}
  The number of plane rooted multi-edge trees of size $n$ is
  \begin{align*}
    A_n=[z^n]T(z)&=[z^n](u+1)=[z^n](2-(1-u))\\
    &=
    [z^n]\biggl(2-\sqrt{5}\sqrt{1-5z}+\frac52(1-5z)+O((1-5z)^{3/2})\biggr).
  \end{align*}
  Singularity analysis yields
  \begin{equation*}
    A_n=-\sqrt{5}\frac{5^nn^{-3/2}}{\Gamma(-1/2)}+O\biggl(\frac{5^n}{n^{5/2}}\biggr)
    =\frac{5^{n+1/2}}{2\sqrt{\pi n^3}}\biggl(1+O\Bigl(\frac1n\Bigr)\biggr).
  \end{equation*}
  Combining this with \eqref{eq:expectation-numerator} yields
  \eqref{eq:expected-height}.
\end{proof}

\subsection{Local Limit Theorem}
In this section, we prove a local limit theorem for the height of a plane
rooted multi-edge tree. As our generating function is very explicit, we can
give a result in a wider range than \cite{Flajolet-Gao-Odlyzko-Richmond:1993}.

\begin{theorem}\label{theorem:local-limit-theorem}Let
  $0<\varepsilon<\frac16$. Then, for
  \begin{equation}\label{eq:llt-range-h}
    \sqrt{\frac{4n\pi^2}{5\varepsilon\log n}}<h<n^{3/4-\varepsilon},
  \end{equation}
  the probability of a plane rooted
  multi-edge tree to have height $h$ is
  \begin{equation*}
    \frac{5h}{n}G\Bigl(\frac{\sqrt{5}h}{2\sqrt{n}}\Bigr)\Bigl(1+O\Bigl(\frac hn
    +\frac{h^4}{n^3}+\frac{\log n}{n^{1/2-2\varepsilon}}\Bigr)\Bigr)
  \end{equation*}
  for
  \begin{equation}\label{eq:poisson}
    \begin{aligned}
      G(\alpha)&=\sum_{m\ge 1}(2\alpha^2 m^2-3)m^2
      \exp(-\alpha^2m^2)\\
      &=\frac{\sqrt{\pi^5}}{\alpha^5}\sum_{m\ge 1}
      \Bigl(2\Bigl(\frac{\pi}{\alpha}\Bigr)^2m^2-3\Bigr)
      m^2\exp\Bigl(-\Bigl(\frac{\pi}{\alpha}\Bigr)^2 m^2\Bigr).
    \end{aligned}
\end{equation}
\end{theorem}
The fact that the two expressions for $G(\alpha)$ in \eqref{eq:poisson} are
equal is Poisson's sum formula (cf.\cite[(3.12.1)]{Bruijn:1958})
for $f(x)=(2\alpha^2 x^2-3)x^2\exp(-\alpha^2 x^2)$.

We first compute the integral which will appear by application of the saddle point method.
\begin{lemma}\label{lemma:integral}
  Let $0<a<1$, $0<b$ be real numbers and $c$, $d$ be complex numbers.
  Then
  \begin{multline*}
    \int_{-\infty}^\infty
    \frac{(ct+d)^3\exp\bigl(-\frac{t^2}5\bigr)}{(1-ae^{ibt})^2}\, dt \\
    =
    \sqrt{5\pi}\sum_{m\ge 0}(m+1)\Bigl(\frac{15}2\Bigl(\frac52
    cibm+d\Bigr)c^2+\Bigl(\frac52 cibm+d\Bigr)^3\Bigr)\\\times a^m\exp\Bigl(-\frac{5}{4}b^2m^2\Bigr).
  \end{multline*}
\end{lemma}
\begin{proof}
  We expand the denominator of the integrand as a binomial series, dominated by
  $(1-a)^{-2}$. Thus
  \begin{multline*}
    \int_{-\infty}^\infty
    \frac{(ct+d)^3\exp\bigl(-\frac{t^2}5\bigr)}{(1-ae^{ibt})^2}\, dt \\=
    \sum_{m\ge 0}(m+1)a^m \int_{-\infty}^\infty (ct+d)^3 \exp\Bigl(-\frac{t^2}{5}+ibmt\Bigr)\,dt.
  \end{multline*}
  Substituting $t=z+\frac{5}{2}ibm$ and shifting the path of integration back
  to the real line yields
  \begin{multline*}
    \int_{-\infty}^\infty(ct+d)^3
    \exp\Bigl(-\frac{t^2}{5}+ibmt\Bigr)\,dt\\
    = \exp\Bigl(-\frac{5}{4}b^2m^2\Bigr)\sqrt{5\pi}\Bigl(\frac{15}2\Bigl(\frac{5}{2}cibm+d\Bigr)c^2+\Bigl(\frac{5}{2}cibm+d\Bigr)^3\Bigr).
  \end{multline*}
\end{proof}

\begin{proof}
  Instead of computing the number of trees of height exactly $h$, we compute
  the number $A_{nh}$ of trees of height exactly $h-1$ because this leads to
  more convenient formul\ae{} and does not matter asymptotically. By
  \eqref{eq:T-T_h-substituted-u}, we get
  \begin{align*}
    A_{nh}:=[z^n](T_{h-1}-T_{h-2})(z)&=[z^n]((T-T_{h-2})-(T-T_{h-1}))\\
    &=[z^n]\frac{1-u^2}{u}\Bigl(\frac{u^{h}}{1-u^h}-\frac{u^{h+1}}{1-u^{h+1}}\Bigr)\\
    &=[z^n]\frac{1-u^2}{u}\frac{u^h-u^{h+1}}{(1-u^h)(1-u^{h+1})}\\
    &=[z^n]\frac{(1+u)(1-u)^2}{u}\frac{u^h}{(1-u^h)(1-u^{h+1})}
  \end{align*}
  for $z=\zeta(u)$.

  Using this transformation
  and Cauchy's formula as in the proof of
  Theorem~\ref{theorem:explicit-formula-height} yields
  \begin{align*}
    A_{nh}&=\frac{1}{2\pi i}\oint_{\abs z\text{ small}} \frac{(1+u)(1-u)^2}{u}\frac{u^h}{(1-u^h)(1-u^{h+1})}\frac{dz}{z^{n+1}}\\
    &=\frac{1}{2\pi i}\oint_{\abs{u}\text{ small}}\frac{(1+u)^2(1-u)^3(u^2+3u+1)^{n-1}}{(1-u^h)(1-u^{h+1})u^{n-h+2}}\,du.
  \end{align*}
Now we apply the saddle point method to this integral. It turns out that the right choice for the contour of integration is given by the parametrisation $u=re^{i\varphi}$ with
  $r=\exp\bigl(-\frac52\frac{h}{n}\bigr)$ and $-\pi\le\varphi\le \pi$. This
  yields
  \begin{align*}
    A_{nh}&=\frac{1}{2\pi}\int_{-\pi}^\pi\frac{(1+u)^2(1-u)^3(u^2+3u+1)^{n-1}}{(1-u^h)(1-u^{h+1})u^{n-h+1}}\,d\varphi\\
    &= \frac{1}{2\pi}\int_{-\pi}^\pi \frac{g(u)\exp(nf(u))}{(1-u^h)(1-u^{h+1})}\,d\varphi
  \end{align*}
  for
  \begin{align*}
    f(u)&=\log(1+3u+u^2)+\Bigl(\frac{h}{n}-1\Bigr)\log u,\\
    g(u)&=\frac{(1+u)^2(1-u)^3}{u(u^2+3u+1)}.
  \end{align*}
  We set 
  \begin{equation*}
    \alpha^2=\frac{5h^2}{4n}.
  \end{equation*}
  By the assumption \eqref{eq:llt-range-h}, we have
  \begin{equation}\label{eq:alpha-square-lower-bound}
    \alpha^2>\frac{\pi^2}{\varepsilon \log n}.
  \end{equation}
  Note that $r\to 1$ for $n\to\infty$. We also note that $g(u)=O(1)$ on the area of
  integration. If $\alpha^2\le\pi$,
  \[\abs{1-u^h}\ge 1-r^h= 1-\exp(-2\alpha^2)\ge
  2\alpha^2\exp(-2\alpha^2)\ge
  2\alpha^2\exp(-2\pi);\]
  thus
  \begin{equation}\label{eq:limit-theorem-denominator-bounds}
    \frac{1}{1-u^h}=O\Bigl(\frac{n}{h^2}\Bigr)=O(\log n), \qquad
    \frac{1}{1-u^{h+1}}=O\Bigl(\frac{n}{h^2}\Bigr)=O(\log n).
  \end{equation}
  Otherwise, $r^h\le \exp(-2\pi)$, i.e., $\frac{1}{1-u^h}$ and
  $\frac{1}{1-u^{h+1}}$ are bounded. Thus
  \eqref{eq:limit-theorem-denominator-bounds} can be used in any case.

  We first prune the tails. We set $\delta_n=n^{-1/2+\varepsilon}$ such that
  $n\delta_n^2=n^{2\varepsilon}$ and
  $n\delta_n^4=n^{-1+4\varepsilon}\le n^{-1/2+\varepsilon}$ and
  $n\delta_n/h^2=O(n^{-1/2+\varepsilon}\log n)$ for $n\to\infty$. In particular,
  we have $\delta_n=o(1)$.

  For $\abs{\varphi}>\delta_n$, we have
  \begin{align*}
    \abs{1+3u+u^2}&\le \abs{1+3u}+r^2=
    \sqrt{1+6r\cos\varphi+9r^2}+r^2\\
    &\le \sqrt{1+6r\cos\delta_n+9r^2}+r^2\\
    &\le\sqrt{1 + 9r^2 +6r - 6r\frac{\delta_n^2}{3}}+r^2
    =\sqrt{(1+3r)^2-2r\delta_n^2}+r^2\\
    &\le 1+3r+r^2-\frac{r}{1+3r}\frac{\delta_n^2}{2}
    \le 1+3r+r^2-\frac{\delta_n^2}{10}
  \end{align*}
  for sufficiently large $n$.
  We conclude that for $\abs{\varphi}>\delta_n$,
  \begin{equation*}
    \Re f(u)\le \log\Bigl(1+3r+r^2-\frac{\delta_n^2}{10}\Bigr)   +
    \Bigl(\frac{h}{n}-1\Bigr)\log r \le f(r)-\frac{\delta_n^2}{100}
  \end{equation*}
  for sufficiently large $n$.
  Thus, by \eqref{eq:limit-theorem-denominator-bounds},
  \begin{multline*}
    A_{nh} =
    \frac1{2\pi}\int_{-\delta_n}^{\delta_n}\frac{g(u)\exp(nf(u))}{(1-u^h)(1-u^{h+1})}\,d\varphi\\
    + O\Bigl(\log^2n\exp(nf(r))\exp\Bigl(-\frac{n\delta_n^2}{100}\Bigr)\Bigr).
  \end{multline*}
  
  We now approximate the integrand in the central region.
  We have
  \begin{align*}
    f(u)&=\log 5 + \frac{h}{n}\Bigl(-\frac{5h}{2n}+i\varphi \Bigr)
    +\frac{1}{5}\Bigl(-\frac{5h}{2n}+i\varphi \Bigr)^2 \\
    &\qquad\qquad\qquad+ O\Bigl(\Bigl(\frac{h}{n}+|\varphi|\Bigr)^4\Bigr)\\
    &=\log 5 -\frac{5h^2}{4n^2}  - \frac{\varphi^2}{5} + O\Bigl(\Bigl(\frac{h}{n}+|\varphi|\Bigr)^4\Bigr),\\
    g(u) &= \frac45\Bigl(\frac{5h}{2n} - i\varphi\Bigl)^3\Bigl(1+O\Bigl(\frac{h}{n}+|\varphi|\Bigr)\Bigr),\\
    \frac{1-u^{h+1}}{1-u^h}&=1 + \frac{u^h(1-u)}{1-u^h} = 1 +
    O\Bigl(\frac{n}{h^2}\Bigl(\frac{h}{n}+|\varphi|\Bigr)\Bigr)\\
    &=1+O\Bigl(\frac{1}{h}+\frac{n|\varphi|}{h^2}\Bigr).
  \end{align*}
  Therefore, noting that $n (h/n+\delta_n)^4=O(h^4/n^3+ n\delta_n^4)=o(1)$ yields
  \begin{multline*}
    A_{nh}=\frac{2\cdot 5^n\exp\bigl(-\frac{5h^2}{4n}\bigr)}{5\pi} \int_{-\delta_n}^{\delta_n}\frac{\bigl(\frac{5h}{2n}-i\varphi\bigr)^3}{(1-u^h)^2}\exp\Bigl(-\frac{n\varphi^2}{5}\Bigr)\\
    \times \Bigl(1+O\Bigl(\frac hn +\frac{h^4}{n^3}+\frac{\log n}{n^{1/2-\varepsilon}}\Bigr)\Bigr)\,d\varphi\\
    + O\Bigl(5^n\log^2n\exp\Bigl(-\frac{5h^2}{4n} -\frac{n\delta_n^2}{100}\Bigr)\Bigr).
  \end{multline*}
  We now use the substitution $\sqrt{n}\varphi=t$, leading to
  \begin{multline*}
    \frac{A_{nh}5\pi\sqrt{n}}{2\cdot 5^n\exp\bigl(-\frac{5h^2}{4n}\bigr)}
    = \int_{-\delta_n\sqrt{n}}^{\delta_n\sqrt{n}}\frac{\bigl(\frac{5h}{2n}-i\frac{t}{\sqrt{n}}\bigr)^3}{(1-u^h)^2}\exp\Bigl(-\frac{t^2}{5}\Bigr)\\
    \times \Bigl(1+O\Bigl(\frac hn +\frac{h^4}{n^3}+\frac{\log n}{n^{1/2-\varepsilon}}\Bigr)\Bigr)\,dt\\
    + O\Bigl(\sqrt{n}\log^2n\exp\Bigl(-\frac{n\delta_n^2}{100}\Bigr)\Bigr).
  \end{multline*}
  We set
  \begin{align*}
    I_{hn}&=\int_{-\infty}^{\infty}
    \frac{\bigl(\frac{5h}{2n}-i\frac{t}{\sqrt{n}}\bigr)^3}{(1-u^h)^2}\exp\Bigl(-\frac{t^2}{5}\Bigr)\,dt,\\
    E_{hn}&=\frac{1}{(1-r^h)^2}\int_{-\infty}^{\infty}
    \Bigl(\frac{5h}{2n}+\frac{\abs{t}}{\sqrt{n}}\Bigr)^3\exp\Bigl(-\frac{t^2}{5}\Bigr)\,dt,
  \end{align*}
  and note that the contribution of $\abs t>\delta_n\sqrt{n}$ is again negligible: we have
  \begin{multline*}
    \Bigg|\int_{\delta_n\sqrt{n}}^{\infty} \frac{\bigl(\frac{5h}{2n}-i\frac{t}{\sqrt{n}}\bigr)^3}{(1-u^h)^2}\exp\Bigl(-\frac{t^2}{5}\Bigr)\,dt \Bigg| \\\leq \frac{1}{(1-r^h)^2} \int_{\delta_n\sqrt{n}}^{\infty} \biggl(\frac{5h}{2n}+\frac{t}{\sqrt{n}}\biggr)^3 \exp \biggl( -\frac{t \delta_n \sqrt{n}}{5} \biggr)\,dt.
  \end{multline*}
  Now we can use the estimate~\eqref{eq:limit-theorem-denominator-bounds} for
  $1-r^h$ as before, and the integral in the upper bound can in principle be
  computed explicitly. It is $O((\sqrt{n}\delta_n)^{-1} \exp(-n\delta_n^2/5))$,
  so the total contribution of the tails (i.e., the regions where $\abs
  t>\delta_n\sqrt{n}$; of course the estimate for negative $t$ is analogous) is
  $O(n^{-1/2}\delta_n^{-1}\log^2n \exp(-n\delta_n^2/5))$. It would be possible to
  give an even better bound, but this is enough for our purposes.

  We obtain
  \begin{multline}\label{eq:1}
    \frac{A_{nh}5\pi\sqrt{n}}{2\cdot 5^n\exp\bigl(-\frac{5h^2}{4n}\bigr)}
    = I_{hn}+ O\Bigl(E_{hn}\Bigl(\frac hn +\frac{h^4}{n^3}+\frac{\log n}{n^{1/2-\varepsilon}}\Bigr)\Bigr)\\
    + O\Bigl(\sqrt{n}\log^2n\exp\Bigl(-\frac{n\delta_n^2}{100}\Bigr)\Bigr).
  \end{multline}
  By Lemma~\ref{lemma:integral} with $a=\exp\bigl(-(5h^2)/(2n)\bigr)$,
  $b=h/\sqrt{n}$, $c=-i/\sqrt{n}$ and $d=(5h)/(2n)$ and by replacing $m+1$ by
  $m$, we obtain
  \begin{equation}\label{eq:I_h_n_G}
    \begin{aligned}
      I_{hn}&=\frac{25h\sqrt{5\pi}\exp\bigl(\frac{5h^2}{4n}\bigr)}{4n^2}\sum_{m\ge 1}\Bigl(\frac{5 h^{2}}{2n}m^2-3\Bigr)m^2
      \exp\Bigl(-\frac{5h^2}{4n}m^2\Bigr)\\
      &=\frac{25h\sqrt{5\pi}\exp(\alpha^2)}{4n^2}G(\alpha).
    \end{aligned}
  \end{equation}
  The integral $E_{hn}$ can be bounded by
  \begin{equation}\label{eq:limit-theorem-E-bound}
    E_{hn}=O\biggl(\frac{\frac{h^3}{n^3}+\frac{1}{n^{3/2}}}{(1-r^h)^2}\biggr).
  \end{equation}

  We first consider the case $\alpha^2\ge \pi$. In this case, we have
  $E_{hn}=O(h^3/n^3)$. All summands in the first expression in
  \eqref{eq:poisson} are positive and its first summand is at least
  $\alpha^2\exp(-\alpha^2)$, so that 
  \begin{equation*}
    I_{hn}=\Omega\Bigl(\frac{h^3}{n^3}\Bigr)=\Omega(E_{hn}).
  \end{equation*}
  Then \eqref{eq:1} yields
  \begin{equation}\label{eq:probability-case-1}
    \frac{A_{nh}5\pi\sqrt{n}}{2\cdot 5^n\exp\bigl(-\frac{5h^2}{4n}\bigr)}
    = \frac{25h\sqrt{5\pi}\exp(\alpha^2)}{4n^2}G(\alpha)\Bigl(1+ O\Bigl(\frac
    hn +\frac{h^4}{n^3}+\frac{\log n}{n^{1/2-\varepsilon}}\Bigr)\Bigr).
  \end{equation}

  We now turn to the case $\alpha^2<\pi$. We now use the second expression for $G(\alpha)$ in
  \eqref{eq:poisson}. Again, all summands are positive and we bound $G(\alpha)$
  by the first summand from below. This yields
  \begin{equation*}
    G(\alpha)=\Omega\Bigl(\frac{1}{\alpha^{7}} \exp\Bigl(-\frac{\pi^2}{\alpha^2}\Bigr)\Bigr)
  \end{equation*}
  and, by \eqref{eq:I_h_n_G} and \eqref{eq:alpha-square-lower-bound},
  \begin{equation*}
    I_{hn}=\Omega\Bigl(\frac{n^{3/2}}{h^{6}}\exp(-\varepsilon\log n)\Bigr)=\Omega\Bigl(\frac{n^{3/2-\varepsilon}}{h^6}\Bigr).
  \end{equation*}
  For an upper bound of $E_{hn}$, we use the estimate $(1-r^h)^{-1}=O(n/h^2)$,
  cf.\ \eqref{eq:limit-theorem-denominator-bounds}. We get
  \begin{equation*}
    E_{hn}=O\Bigl(\frac{1}{n^{3/2}}\cdot \frac{n^2}{h^4}\Bigr)=O\Bigl(\frac{n^{1/2}}{h^{4}}\Bigr)=
    O\Bigl(\frac{n^{3/2-\varepsilon}}{h^6} \frac{h^{2}}{n} n^{\varepsilon}\Bigr)=O(n^{\varepsilon}I_{hn}).
  \end{equation*}
  Thus \eqref{eq:1} yields
  \begin{equation}\label{eq:probability-case-2}
    \frac{A_{nh}5\pi\sqrt{n}}{2\cdot 5^n\exp\bigl(-\frac{5h^2}{4n}\bigr)}
    = \frac{25h\sqrt{5\pi}\exp(\alpha^2)}{4n^2}G(\alpha)\Bigl(1+
    O\Bigl(\frac{\log n}{n^{1/2-2\varepsilon}}\Bigr)\Bigr).
  \end{equation}
  Combining \eqref{eq:probability-case-1} and \eqref{eq:probability-case-2}
  with \eqref{eq:number-trees} yields the result.
\end{proof}
\subsection{Number of Vertices}
In this section, we consider the number of vertices of a random rooted plane
multi-edge tree of size $n$.

We first give an explicit formula.

\begin{theorem}\label{theorem:fixed-vertices}
  The number of rooted plane multi-edge trees of size $n$ with $k$ vertices is
  \begin{equation}\label{eq:exact-number}
    \frac1k\binom{2k-2}{k-1}\binom{n-1}{k-2}.
  \end{equation}
\end{theorem}
\begin{proof}
 We first provide a proof based on the generating function, which will also be needed later.
Let $T(y,z)$ be the bivariate generating function for rooted plane multi-edge trees, where
$y$ marks the number of vertices and $z$ the number of edges. Rooted plane multi-edge trees
  $\calT$ can be represented symbolically as
  \begin{equation}\label{eq:symbolic-equation-marked}
    \calT = \{y\} \times (\calE^+\calT)^*.
  \end{equation}
 This symbolic equation translates to
  \begin{equation}\label{eq:vertices-bivariate-generating-function-equation-1}
    T(y, z)=\frac{y}{1-\frac{z}{1-z}T(y, z)}=\frac{y(1-z)}{1-z-zT(y, z)}.
  \end{equation}

  For a fixed $z$, we compute the coefficient $[y^k]T(y, z)$ using the Lagrange inversion formula.
  By~\eqref{eq:vertices-bivariate-generating-function-equation-1}, we have
  \begin{equation*}
    y= T(y, z) \frac{1-z-zT(y, z)}{1-z}.
  \end{equation*}
Now the Lagrange inversion formula gives us
  \begin{align*}
    [y^k]T(y, z)&=\frac1k [T^{k-1}]\Bigl(\frac{1-z}{1-z-zT}\Bigr)^{k}\\
    &=\frac1k[T^{k-1}]\Bigl(\frac1{1-\frac{z}{1-z}T}\Bigr)^k
    \\&=\frac1k \binom{-k}{k-1}(-1)^{k-1} \Bigl(\frac{z}{1-z}\Bigr)^{k-1}\\
    &=\frac{1}{k}\binom{2k-2}{k-1}\Bigl(\frac{z}{1-z}\Bigr)^{k-1}.
  \end{align*}
  Finally, we extract the coefficient of $z^n$:
  \begin{align*}
    [z^n][y^k]T(y, z)&=\frac1k\binom{2k-2}{k-1}[z^{n-k+1}](1-z)^{1-k}\\
    &=
    \frac1k\binom{2k-2}{k-1}\binom{1-k}{n-k+1}(-1)^{n-k+1}\\
    &=\frac1k\binom{2k-2}{k-1}\binom{n-1}{n-k+1}.
  \end{align*}
\end{proof}
\begin{proof}[Combinatorial proof of
  Theorem~\ref{theorem:fixed-vertices}]
  It is well known that the number of plane rooted trees (without multiple
  edges) with $k$ vertices is given by the Catalan number
  $C_{k-1}=\frac1k\binom{2k-2}{k-1}$. Each such tree can be transformed into a
  multi-edge tree of size $n$ by distributing the $n$ edges among the $k-1$
  edges of the non-multi-edge tree. This corresponds to a composition of $n$
  into $k-1$ parts. There are $\binom{n-1}{k-2}$ such compositions. Thus there
  are $\frac1k\binom{2k-2}{k-1}\binom{n-1}{k-2}$ plane rooted multi-edge tree of
  size $n$ with $k$ vertices.
\end{proof}

The distribution of the number of vertices can now be derived from the
explicit formula in Theorem~\ref{theorem:fixed-vertices} using Stirling's
formula. In order to determine the asymptotic behaviour of the moments, we use an approach via Hwang's quasi power theorem which turns out to
be more convenient.
\begin{theorem}\label{theorem:number-vertices}
  Let $V_n$ be the number of vertices of a random rooted plane multi-edge tree
  of size $n$. Then
  \begin{align*}
    \E(V_n)&=\frac{4}{5}n+\frac{9}{10}+O\left(\frac{1}{n}\right),\\
    \V(V_n)&=\frac{4}{25}n + \frac{2}{25}+O\left(\frac{1}{n}\right),\\
\intertext{and}
    \P\left(\frac{V_n-\frac45n}{\frac25\sqrt{n}}\le
      v\right)&=\frac1{\sqrt{2\pi}}\int_{-\infty}^v e^{-t^2/2}\,dt + O\left(\frac{1}{\sqrt{n}}\right)
  \end{align*}
  holds uniformly for $v\in\R$. Furthermore, the local limit theorem 
  \begin{equation*}
    \P(V_n=k)\sim \frac{5}{2\sqrt{2 n\pi}}\exp\Bigl(-\frac{1}{2}\Bigl(\frac{k-\frac{4}{5}n}{\frac{2}{5}\sqrt{n}}\Bigr)^2\Bigr)
  \end{equation*}
  holds for $k = \frac{4n}{5} + o(n^{2/3})$.
\end{theorem}
\begin{proof}
Let $T(y,z)$ be the bivariate generating function as in the first proof of Theorem~\ref{theorem:fixed-vertices}. The functional equation~\eqref{eq:vertices-bivariate-generating-function-equation-1} is equivalent to
  \begin{equation}\label{eq:vertices-bivariate-generating-function-equation}
    zT(y, z)^2-(1-z)T(y, z)+y(1-z)=0.
  \end{equation}
  Solving this quadratic equation yields
  \begin{equation}\label{eq:bivariate-generating-function}
    T(y, z)=\frac{(1-z)-\sqrt{1-z}\sqrt{1-(4y+1)z}}{2z};
  \end{equation}
  note that the negative sign has to be chosen to obtain regularity at $z=0$.

  The probability generating function of $V_n$ is then
  \begin{equation*}
    p_n(y)=\frac{[z^n]T(y, z)}{[z^n]T(1, z)}.
  \end{equation*}

  For $y$ in a neighbourhood of $1$, the dominant singularity is at
  $z=1/(1+4y)$. As
  \begin{align*}
    T(y,
    z)&=\frac{1-\frac1{1+4y}}{\frac{2}{1+4y}}-\frac{\sqrt{1-\frac{1}{1+4y}}}{\frac{2}{1+4y}}\sqrt{1-(4y+1)z}
    + O(1- (4y+1)z)\\
    &=
    2y - \sqrt{y(1+4y)}\sqrt{1-(4y+1)z} + O(1- (4y+1)z)
  \end{align*}
  for $z\to 1/(1+4y)$ except for one ray, singularity analysis
  \cite{Flajolet-Odlyzko:1990:singul} yields
  \begin{align*}
    [z^n]T(y, z) &= -\frac{\sqrt{y(1+4y)}}{\Gamma(-1/2)}(4y+1)^{n}n^{-3/2} +
    O((4y+1)^n n^{-5/2})  \\
    &=\frac{\sqrt{y(1+4y)}}{2\sqrt{\pi n^3}}(4y+1)^n\Bigl(1+O\Bigl(\frac1n\Bigr)\Bigr).
  \end{align*}
  For $y=1$, this coincides with \eqref{eq:number-trees}.
  
  Thus
  \begin{equation*}
    p_n(y)=\frac{[z^n]T(y, z)}{[z^n]T(1, z)}=\sqrt{y}\left(\frac{4y+1}{5}\right)^{n+1/2}\left(1+O\left(\frac1n\right)\right).
  \end{equation*}
  The asymptotic formul\ae{} for mean and variance in Theorem~\ref{theorem:number-vertices} as well as the central limit theorem are an immediate
  consequence of Hwang's quasi power theorem~\cite{Hwang:1998} in the version
  of \cite[Theorem~IX.8]{Flajolet-Sedgewick:ta:analy}.

The local limit theorem follows immediately from the explicit formula in~Theorem \ref{theorem:fixed-vertices}: applying Stirling's formula to~\eqref{eq:exact-number}, we find that the total number of multi-edge trees with $n$ edges and $k = \frac{4n}{5} + R$ vertices is equal to
$$\frac{5^{n+3/2}}{4\pi \sqrt{2} n^2} \exp \Big( - \frac{25R^2}{8n} + O \Big( \frac{1}{n} + \frac{R}{n} + \frac{R^3}{n^2} \Big) \Big).$$
Combining this with the asymptotic formula~\eqref{eq:number-trees} for the
total number $A_n$ of multi-edge trees with $n$ edges, we obtain the desired statement for $R = o(n^{2/3})$.
\end{proof}

\bibliographystyle{amsplainurl}
\bibliography{bib/cheub}

\end{document}